\documentclass[11pt,fullpage]{article}
\usepackage{multicol,wrapfig,amsmath,subfigure}
\usepackage{epstopdf}
\usepackage{multirow}
\usepackage{amsmath}
\usepackage{amsfonts,amssymb}
\usepackage{amsthm}
\usepackage{graphics,graphicx}

\usepackage[hyphens,spaces,obeyspaces]{url}

\usepackage{hyperref}
\usepackage[right = 2.5cm, left=2.5cm, top = 2.5cm, bottom =2.5cm]{geometry}
\usepackage{tikz-cd}
\usepackage{slashed}
\usepackage{enumerate}
\usepackage{comment}
\usepackage{authblk}
\usepackage{mathtools}
\usepackage{relsize}
\usepackage[bbgreekl]{mathbbol}
\usepackage{amsfonts}
\DeclareSymbolFontAlphabet{\mathbb}{AMSb} 
\DeclareSymbolFontAlphabet{\mathbbl}{bbold}
\newcommand{\Prism}{{\mathlarger{\mathbbl{\Delta}}}}

\newtheorem{prop}{Proposition}[section]
\newtheorem{lemma}{Lemma}[section]
\newtheorem{defn}{Definition}[section]
\newtheorem{thm}{Theorem}[section]
\newtheorem{cor}{Corollary}[section]

\newtheorem{example}{Example}[section]

\newcommand{\C}{\mathbb{C}}
\newcommand{\Z}{\mathbb{Z}}
\newcommand{\Q}{\mathbb{Q}}

\newcommand{\etale}{\'etale }
\newcommand{\Etale}{\'Etale }
\newcommand{\Einf}{$E_{\infty}$}
\newcommand{\Fp}{$\mathbb{F}_p$}
\newcommand{\Fpbar}{$\overline{\mathbb{F}_p}$}
\newcommand{\Ainf}{$A_{\textnormal{inf}}$ }

\newcommand{\FpCochains}{$C^{*}_{\textnormal{sing}}(X(\C), \mathbb{F}_p)$ }
\newcommand{\FpEtaleCochains}{\textnormal{R}$\Gamma_{\textnormal{\'et}}(X, \mathbb{F}_p)$ }
\newcommand{\FpbarFunctor}{$C^{*}_{\textnormal{sing}}(-, \overline{\mathbb{F}_p} )$ }

\newcommand{\RG}{\textnormal{R}$\Gamma$}

\graphicspath{ {images/} }

\begin{document}

\title{Prismatic cohomology and $p$-adic homotopy theory}
\author{Tobias Shin}

\begin{center}
\Large\textbf{Prismatic cohomology and $p$-adic homotopy theory}
\linebreak
\text{\normalsize{Tobias Shin}}\\
\end{center}
\begin{abstract}
This short note regards an observation about the recent theory of prismatic cohomology developed by Bhatt and Scholze. In particular, by applying a functor of Mandell, we see that the \etale comparison theorem in the prismatic theory reproduces the $p$-adic homotopy type for a smooth proper complex variety with good reduction mod $p$.
\end{abstract}
\section{Introduction}
The recent theory of prismatic cohomology developed by Bhatt and Scholze \cite{bhatt} unifies well known $p$-adic cohomology theories, such as de Rham, \'etale, and crystalline, into one general framework through various comparison theorems. As an application, the work of Bhatt, Morrow, and Scholze in \cite{BMS} yields the following corollary of the \etale comparison theorem in \cite{bhattEilenberg} (Lecture IX, Theorem 5.1).
\begin{thm}\label{BMSetale} (Bhatt-Morrow-Scholze) Let $\C_p$ denote the $p$-completion of an algebraic closure of $\Q_p$ and let $\mathcal{O}_{\C_p}$ denote its ring of integers. Let $\frak{X}$ denote a smooth, proper formal scheme over $\mathcal{O}_{\C_p}$ and let $X$ denote its fiber over $\C_p$. Then there is a quasi-isomorphism 
\begin{equation*}
(\textnormal{R}\Gamma(\frak{X},\Prism_{\frak{X}/A_{\textnormal{inf}}})\otimes^{L}_{A_{\textnormal{inf}}} \mathbb{C}^{\flat}_p)^{\phi=1} \simeq \textnormal{R}\Gamma_{\textnormal{\'et}}(X,\mathbb{F}_p)
\end{equation*} 
where the notation $A_{\textnormal{inf}} := W(\mathcal{O}_{\C_p}^\flat)$ denotes the Witt vectors of the tilt of $\mathcal{O}_{\C_p}$, and $(-)^{\phi=1}$ denotes the homotopy fixed points of the lift of Frobenius $\phi$ on the prism $($\Ainf$,\textnormal{ker}($\Ainf $ \xrightarrow{\theta}\mathcal{O}_{\C_p}))$.
\\ \\
In particular, identify $\C_p$ with $\C$ and let $X \subset \mathbb{P}^n$ be a smooth projective variety over $\C$ that arises as the generic fiber of a smooth, proper scheme $X_{\mathbb{Z}}$ over $\Z$, e.g., one whose defining homogeneous polynomials have coefficients in $\Z$ and whose Jacobian matrix has rank $n + 1 - \textnormal{dim}(X)$ modulo $p$. Then the above quasi-isomorphism holds after taking $\frak{X}$ to be the fiber of $X_{\mathbb{Z}}$ over \Ainf.
\end{thm}
In fact, the proof of the above theorem proves more: there are natural \Einf-\Fp-\textit{algebra} structures on the cochains above, such that the quasi-isomorphism is a map of \Einf-\Fp-algebras. The natural \Einf-\Fp-algebra structures in discussion arise via Godement resolutions. The theorem above then states one can recover the full \Einf-algebra structure on the \etale \Fp-cochains of a smooth proper complex variety with good reduction mod $p$, by natural constructions applied to the prismatic complex over the prism $($\Ainf$,\textnormal{ker}(\theta))$. On the other hand, we have the following theorem of Mandell \cite{mandellEinf} (see also \cite{mandellEinf}, Remark 5.1).

\begin{thm}\label{mandellTheorem1} (Mandell) Let $\frak{H}$ denote the homotopy category of connected $p$-complete nilpotent spaces of finite $p$-type, and let $h\overline{\mathfrak{E}}$ denote the homotopy category of \Einf-\Fpbar-algebras. The singular cochain functor \FpbarFunctor induces a contravariant equivalence from $\frak{H}$ to a full subcategory of $h\overline{\mathfrak{E}}$. The quasi-inverse on the subcategory is given by $\overline{\mathbb{U}}$ the right derived functor of the functor $A \mapsto \textnormal{Hom}_{\overline{\mathfrak{E}}}(A,C^{*}(\Delta[n],\overline{\mathbb{F}_p}))$. Moreover, there is the following adjunction
\begin{equation*}
[X, \overline{\mathbb{U}}A] \cong [A, C^{*}_{\textnormal{sing}}(X,\overline{\mathbb{F}_p})]
\end{equation*} 
where $[-,-]$ denotes morphisms in the respective homotopy category. Moreover, for $X$ connected and of finite $p$-type, the natural map $X \rightarrow \overline{\mathbb{U}}C^{*}_{\textnormal{sing}}(X, \overline{\mathbb{F}_p})$ via the adjunction is naturally isomorphic to  $p$-completion in the sense of Bousfield-Kan.
\end{thm}
There is also the analog of the above theorem (\cite{mandellEinf}, Proposition A.2, A.3) for the singular cochain functor with coefficients in \Fp.
\begin{thm}\label{mandellTheorem2} (Mandell) Let $h\mathfrak{E}$ denote the homotopy category of \Einf-\Fp-algebras. The right derived functor $\mathbb{U}$ of the functor $A \mapsto \textnormal{Hom}_{\mathfrak{E}}(A,C^{*}(\Delta[n],\mathbb{F}_p))$ from the category of \Einf-\Fp-algebras to simplicial sets is right adjoint to the right derived functor of the singular cochain functor $C^{*}_{\textnormal{sing}}(-,\mathbb{F}_p)$, such that there is a natural isomorphism $LX \rightarrow \mathbb{U}C^{*}_{\textnormal{sing}}(X, \mathbb{F}_p)$ in the homotopy category for $X$ connected, $p$-complete, nilpotent, and of finite $p$-type, where $LX$ denotes the free loop space of $X$.  
\end{thm}

In other words, taking coefficients in $\mathbb{F}_p$ recovers the free loop space of the $p$-adic homotopy type. An immediate consequence of the theorems above is then the following trivial observation:

\begin{thm}\label{MAINTHEOREM} Let $X$ be as in the hypotheses of Theorem \ref{BMSetale}. Assume further that $X$ is nilpotent and finite $p$-type. Then $\mathbb{U}(\textnormal{R}\Gamma(\mathfrak{X},\Prism_{\mathfrak{X}/A_{\textnormal{inf}}})\otimes^{L}_{A_{\textnormal{inf}}} C^{\flat})^{\phi=1} $  is the free loop space of the Bousfield-Kan $p$-completion of the complex variety $X$. Similarly, $\overline{\mathbb{U}}((\textnormal{R}\Gamma(\mathfrak{X},\Prism_{\mathfrak{X}/A_{\textnormal{inf}}})\otimes^{L}_{A_{\textnormal{inf}}} C^{\flat})^{\phi=1} \otimes^L_{\mathbb{F}_p} \overline{\mathbb{F}_p})$ is the Sullivan $p$-completion of the complex variety $X$.
\end{thm}

Recall that for simply connected CW-complexes of finite type, the Bousfield-Kan $p$-completion and Sullivan $p$-completion coincide. So as a corollary we have

\begin{cor} Let $X$ be as in the hypotheses of Theorem \ref{BMSetale}. Assume $X$ is simply connected and finite $p$-type. Then $\mathbb{U}(\textnormal{R}\Gamma(\mathfrak{X},\Prism_{\mathfrak{X}/A_{\textnormal{inf}}})\otimes^{L}_{A_{\textnormal{inf}}} C^{\flat})^{\phi=1} $ is the free loop space of the $p$-completion of $X$, and  $\overline{\mathbb{U}}((\textnormal{R}\Gamma(\mathfrak{X},\Prism_{\mathfrak{X}/A_{\textnormal{inf}}})\otimes^{L}_{A_{\textnormal{inf}}} C^{\flat})^{\phi=1} \otimes^L_{\mathbb{F}_p} \overline{\mathbb{F}_p})$ is the $p$-completion of $X$.
\end{cor}

That is, for a smooth proper variety over $\C$ with good reduction mod $p$, natural constructions applied to its prismatic complex over \Ainf yield its $p$-adic homotopy type (in the senses above). 
\\ \\
The purpose of this note covers the relevant comparisons between \etale cochains and singular cochains needed to prove this observation. It is the opinion of the author that these comparisons are well known to experts but somewhat difficult to find stated explicitly in the literature.
\\ \\
\textbf{Notation and conventions}. All varieties over $\C$ are assumed connected in the analytic topology. We use $k$ to denote the ground field, usually $\mathbb{F}_p$ or \Fpbar. All \Einf-$k$-algebras are chain complexes of $k$-modules with an action of a fixed \Einf-operad $\mathcal{E}_k$ in $\textnormal{Ch}(k\textnormal{-mod})$, where $\mathcal{E}_k$ has a fixed map of operads $\mathcal{E}_k \rightarrow \mathcal{Z}_k$ to the Eilenberg-Zilbur operad $\mathcal{Z}_k$. We refer the proof that such a map of operads always exists to \cite{mandellEinf}. Often we will omit the subscript $k$ if it is clear in context. By a \textit{quasi-isomorphism of \Einf-algebras}, we mean a morphism of \Einf-algebras that induces a quasi-isomorphism on the underlying chain complexes. Given a site $\mathcal{C}$, we write $\textnormal{AbSh}(\mathcal{C})$ for the category of abelian sheaves on $\mathcal{C}$. We write $\textnormal{Sh}(\mathcal{C})$ and $\textnormal{PSh}(\mathcal{C})$ for the category of sheaves and presheaves of sets respectively on $\mathcal{C}$, and $\textnormal{Sh}(\mathcal{C},D)$ for sheaves with values in a category $D$. $\textnormal{\textbf{Ab}}$ denotes the category of abelian groups, and $D\textnormal{AbSh}(\mathcal{C})$ and $D\textnormal{\textbf{Ab}}$ denote the respective derived categories. If $C$ is some category, then $\Delta C$ denotes the category of cosimplicial objects of $C$. We avoid any $\infty$-categorical language because the author does not know it. 
\\ \\
\textbf{Acknowledgements}. The author would like to thank Joana Cirici and Dan Petersen for helpful comments about \Einf-algebras.

\section{\Etale cochains and \Einf-Artin comparison}\label{section2}
In this mainly expository section, we review some classical theorems regarding the \Einf-algebra structure on \etale cochains. We emphasize there are no original theorems proven in this section; many of the arguments can be found in, for example, \cite{chataur}, \cite{petersen}, and \cite{rodriguez}, with the main conceptual ideas originating from Godement \cite{godement}. The aim of this section is to cover the following well known theorem.

\begin{thm}\label{maincomparison} Let $X$ be a smooth proper complex variety. There is a quasi-isomorphism of \Einf-algebras between the singular \Fp-cochains \FpCochains and the \etale \Fp-cochains {\textnormal{R}$\Gamma_{\textnormal{\'et}}(X, \mathbb{F}_p)$}.
\end{thm} 

Here is an outline of the section: first, we compare the \etale site of $X$ over $\C$ with the site of analytic open sets on its underlying complex manifold $X(\C)$, by passing to the site of local homeomorphisms mapping to $X(\C)$. We analyze the site of local homeomorphisms and show it has enough points. We then use the Godement resolution on all three sites to obtain \Einf-algebras on their sheaf cohomologies; the quasi-isomorphism of the underlying complexes is omitted.
\\ \\

We have the following theorem from SGA IV (XII-4) \cite{SGA4}; we provide a translation of part of the proof, with some details provided using lemmas from the Stacks Project \cite{stacks}.

\begin{thm}\label{zigzagsites} Let $X$ be a smooth proper complex variety. There is a zig-zag of sites
\begin{center}
\begin{tikzcd}[column sep = small, row sep = small]
& X_{\textnormal{cl}} \arrow[dl, "\delta" above left] \arrow[dr, "\varepsilon" above right] \\ 
X(\C) & & X_{\textnormal{\'{e}t}}
\end{tikzcd}
\end{center}
where $X(\C)$ denotes the site of analytic open sets on the underlying complex manifold of $X$, $X_{\textnormal{cl}}$ denotes the site of local homeomorphisms $U \rightarrow X(\C)$, and $X_{\textnormal{{\'{e}t}}}$ denotes the \etale site of $X$. 
\end{thm}
\begin{proof} An object of $X_{\textnormal{cl}}$ is a continous map of topological spaces $f: U \rightarrow X(\C)$ such that for every point $x \in U$, there is a neighborhood $U_x$ such that the restriction of $f$ to $U_x$ is a homeomorphism onto an open neighborhood around $f(x)$; that is, $f$ is a local homeomorphism. Since inclusions of open sets in $X(\C)$ are local homeomorphisms, we obtain a morphism of sites $\delta: X_{\textnormal{cl}} \rightarrow X(\C)$ by the continuous functor $U \mapsto (U \hookrightarrow X(\C))$; this continuous functor is the inclusion of categories $X(\C) \subset X_{\textnormal{cl}}$. 
\\ \\
On the other hand, let $f: X' \rightarrow X$ be \'etale. Then the induced map on the underlying smooth manifolds $f(\C): X'(\C) \rightarrow X(\C)$ is a local isomorphism, by the Jacobian criterion and implicit function theorem. The functor $X' \mapsto X'(\C)$ then induces the morphism of sites $\varepsilon: X_{\textnormal{cl}} \rightarrow X_{\textnormal{\'et}}$. $\qed$
\end{proof}

\begin{lemma}\label{deltareflects} The functor $\delta_{*}$ sends surjective maps of sheaves of sets to surjective maps of sheaves of sets. Moreover, it is an equivalence of the associated topoi, and reflects injections and surjections.
\end{lemma}
\begin{proof} For each local homeomorphism $f: U \rightarrow X(\C)$, for each $x \in U$, there is a neighborhood homeomorphic to an open neighborhood around $f(x)$ in $X(\C)$. Thus, we can cover $U$ by open sets that are homeomorphic to open sets in $X(\C)$; that is, there exists a family of open sets $\{U_i \hookrightarrow X(\C)\}$ such that we have a commutative diagram of local homeomorphisms

\begin{center}
\begin{tikzcd}
U_i \arrow[rr, hookrightarrow] \arrow[dr, hookrightarrow] & & U \arrow[dl] \\
& X(\C)
\end{tikzcd}
\end{center}
and where $U$ is the union of the images of the maps from $U_i$. The above geometric argument immediately implies that the hypotheses of (\cite{stacks}, Tag 04D5, Lemma 7.41.2) are satisfied, and so $\delta_{*}$ sends surjective maps of sheaves to surjective maps of sheaves. Similarly, the hypotheses of (\cite{stacks}, Tag 04D5, Lemma 7.41.4) are also satisfied, so $\delta_{*}$ reflects injections and surjections.
\\ \\
Lastly, to show $\delta_{*}$ is an equivalence of topoi, notice that the inclusion functor $X(\C) \hookrightarrow X_{\textnormal{cl}}$ is cocontinuous, and that the hypotheses of (\cite{stacks}, Tag 039Z, Lemma 7.29.1) are likewise satisfied by the above geometric argument. The morphism of topoi $g: \textnormal{Sh}(X(\C)) \rightarrow \textnormal{Sh}(X_{\textnormal{cl}})$ associated to the inclusion as a cocontinuous functor is then an equivalence, with the adjunction mappings $g^{-1}g_{*}\mathcal{F} \rightarrow \mathcal{F}$ and $\mathcal{G} \rightarrow g_{*}g^{-1}\mathcal{G}$ being isomorphisms. It follows from the definition of induced morphism of topoi from a cocontinuous functor that $\delta_{*} = g^{-1}$; in fact, since $g^{-1}$ is left adjoint to $g_{*}$, this shows $\delta_{*}$ is right exact as well. 
\end{proof}
\textbf{Remark.} The functor $\delta^{-1}$ is \textit{not} the same as the induced functor $g_{*}$ in the argument above; the map of topoi $(\delta^{-1}, \delta_{*})$ is induced from the inclusion as a \textit{continuous} functor, whereas the map of topoi $(g^{-1}, g_{*})$ is induced from the inclusion as a \textit{cocontinuous} functor. 
\\ \\
The above theorem and lemma say that one can replace $X(\C)$ with $X_{\textnormal{cl}}$ for calculation of usual sheaf cohomology. Moreover there is the following property of the morphism $\varepsilon$, again explained in SGA IV (XII-4).

\begin{thm} Let $X$ be smooth and proper over $\C$. There is an equivalence of categories given by the quasi-inverse functors $\varepsilon_{*}$ and $\varepsilon^{*}$ between the category of locally constant constructible torsion sheaves on $X_{\textnormal{\'et}}$, and the category of locally constant finite fiber torsion sheaves on $X_{\textnormal{cl}}$.
\end{thm}

The proof of the above theorem is what ultimately gives the desired quasi-isomorphism between sheaf cohomologies 
\begin{equation*}
\textnormal{R}\Gamma(X_{\textnormal{\'et}},\mathbb{F}_p) \xrightarrow{\simeq} \textnormal{R}\Gamma(X_{\textnormal{cl}},\mathbb{F}_p) \xleftarrow{\simeq} \textnormal{R}\Gamma(X(\C),\mathbb{F}_p)
\end{equation*}
where the first isomorphism is from $\varepsilon_{*}$ in the theorem above, and the second is by $\delta_{*}$ from Theorem \ref{zigzagsites}. The quasi-isomorphism is really the difficult part of Theorem \ref{maincomparison}; the rest of this section is to show that the above zig-zag of maps are maps of \Einf-\Fp-algebras. First, we show how we obtain induced maps on cohomology from morphisms of sites:

\begin{lemma}\label{derivedcomposition} Let $f: \mathcal{C} \rightarrow \mathcal{C}'$ be a morphism of sites. The following diagram commutes:

\begin{center}
\begin{tikzcd}
D\textnormal{AbSh}(\mathcal{C}) \arrow[r, "\textnormal{R}f_{*}"] \arrow[dr, "\textnormal{R}\Gamma" below left] & D\textnormal{AbSh}(\mathcal{C}') \arrow[d, "\textnormal{R}\Gamma"] \\
& D\textnormal{\textbf{Ab}}
\end{tikzcd}
\end{center}
In particular, we have \RG$(\mathcal{C'}, \textnormal{R}f_{*}\mathcal{F}) \simeq $ \RG$(\mathcal{C},\mathcal{F})$ for any $\mathcal{F}$ in $\textnormal{Sh}(\mathcal{C})$.  
 
\end{lemma}
\begin{proof}
This follows from the following commutative diagram:

\begin{center}
\begin{tikzcd}
\textnormal{AbSh}(\mathcal{C}) \arrow[r, "f_{*}"] \arrow[dr, "\Gamma" below left] & \textnormal{AbSh}(\mathcal{C}') \arrow[d, "\Gamma"] \\
& \textnormal{\textbf{Ab}}
\end{tikzcd}
\end{center}
We have that $f_{*}$ is right adjoint to an exact functor $f^{-1}$. Hence, we have the following string of natural isomorphisms:
\begin{equation*}
\begin{split}
 \Gamma(\mathcal{C}', f_{*}\mathcal{F}) &:=  \textnormal{Hom}_{\textnormal{PSh}(\mathcal{C}')}(*_{\mathcal{C}'}, f_{*}\mathcal{F}) \\
&\cong \textnormal{Hom}_{\textnormal{PSh}(\mathcal{C})}(f^{-1}(*_{\mathcal{C}'}), \mathcal{F}) \\ 
&\cong  \textnormal{Hom}_{\textnormal{PSh}(\mathcal{C})}(*_{\mathcal{C}},\mathcal{F}) \\
&=: \Gamma(\mathcal{C}, \mathcal{F})
\end{split}
\end{equation*}
where the second isomorphism holds since $f^{-1}$ commutes with finite limits, and the terminal object is a limit over an empty diagram. All functors in the above diagram are left exact but not right exact. Since $f_{*}$ preserves injective objects (\cite{stacks}, Tag 015Z), it follows by (\cite{stacks}, Tag 015L, Lemma 13.22.1) that the natural morphism $\textnormal{R}(\Gamma\circ f_{*}) \rightarrow \textnormal{R}\Gamma\circ \textnormal{R}f_{*}$ is an isomorphism. 
\end{proof}

Let us briefly recall the Godement construction \cite{godement}, considered in a general categorical context. Here, we use the language of Rodr\'iguez-Gon\'alez and Roig \cite{rodriguez2}. 

\begin{defn} Let $\mathcal{C}$ be a site. A \textbf{point of the site} $\mathcal{C}$ is a pair of adjoint functors $x = (x^{*},x_{*})$ 

\begin{equation*}
\textnormal{Sh}(\mathcal{C}) \overset{x^{*}}{\underset{x_{*}}\rightleftarrows}  \textnormal{Set},
\end{equation*}
\begin{equation*}
\textnormal{Hom}_{\textnormal{Set}}(x^{*}\mathcal{F},E) \cong \textnormal{Hom}_{\textnormal{Sh}(\mathcal{C})}(\mathcal{F},x_{*}E)
\end{equation*}
 
where $x^{*}$ commutes with finite limits. The site $\mathcal{C}$ \textbf{has enough points} if there exists a set $S$ of points $x$ such that a morphism $f$ in $\textnormal{Sh}(\mathcal{C})$ is an isomorphism if and only if $x^{*}f$ is a bijection for all $x \in S$. 
\end{defn} 
The right adjoint $x_{*}$ assigns to each set $E$ the \textit{skyscraper sheaf} $x_{*}E$ at the point $x$. The left adjoint $x^{*}$ assigns to each sheaf $\mathcal{F}$ the \textit{stalk} $\mathcal{F}_x := x^{*}\mathcal{F}$ at the point $x$.

\begin{example} Consider the site $X(\C)$ of (analytic) open subsets of the complex manifold $X(\C)$. Take $S$ to be the underlying set of $X(\C)$. The geometric points $x \in S$ determine points of the site by taking $x^{*}$ to be the stalk functor and $x_{*}$ to be the skyscraper sheaf functor. Moreover, a map of sheaves is an isomorphism if and only if it is so on the stalks. Thus the site $X(\C)$ has enough points.
\end{example}

\begin{example} It is a classical fact that the \etale site $X_{\textnormal{\'et}}$ has enough points; this follows from, for example, Deligne's criterion (\cite{SGA4}, VI Appendix, Proposition 9.0).
\end{example}

\begin{lemma}\label{Xclenoughpoints} The site $X_{\textnormal{cl}}$ has enough points.
\end{lemma}
\begin{proof} Recall we have an equivalence of topoi, with the equivalence induced by an inclusion of categories, with induced functor $\delta_{*}: \textnormal{Sh}(X_{\textnormal{cl}}) \rightarrow \textnormal{Sh}(X(\C))$ with quasi-inverse $g_{*}$. In fact, $\delta_{*}$ is left adjoint to $g_{*}$. By the example above, we know $X(\C)$ has enough points. For each point $x = (x^{*},x_{*})$ of $X(\C)$, define $y = (y^{*},y_{*}) := (x^{*}\circ\delta_{*}, g_{*}\circ x_{*})$. We check adjointness:
\begin{equation*}
\begin{split}
\textnormal{Hom}_{\textnormal{Set}}(y^{*}\mathcal{F},E) &= \textnormal{Hom}_{\textnormal{Set}}(x^{*}\circ\delta_{*}\mathcal{F},E) \\
& \cong \textnormal{Hom}_{\textnormal{Sh}(X(\C))}(\delta_{*}\mathcal{F},x_{*}E) \\
& \cong \textnormal{Hom}_{\textnormal{Sh}(X_{\textnormal{cl}})}(\mathcal{F}, g_{*}\circ x_{*}E) \\
& = \textnormal{Hom}_{\textnormal{Sh}(X_{\textnormal{cl}})}(\mathcal{F}, y_{*}E)
\end{split}
\end{equation*}
where all isomorphisms follow from the adjunction pairs $(x^{*},x_{*})$ and $(\delta_{*},g_{*})$. Moreover, since $x^{*}$ and $\delta_{*}$ are both exact, their composition $y^{*}$ also commutes with finite limits.
\\ \\
The above argument shows that $y$ as defined is a point. We take as our set of points $S$ a set with cardinality at most the cardinality of the underlying set of the manifold $X(\C)$, via $(x^{*},x_{*}) \mapsto (x^{*}\circ\delta_{*},g_{*}\circ x_{*})$. We now show we can check isomorphisms of sheaves stalk-wise.
\\ \\
Assume $f$ is an isomorphism in $\textnormal{Sh}(X_{\textnormal{cl}})$. By exactness of $\delta_{*}$, we have that $\delta_{*}f$ is an isomorphism in $\textnormal{Sh}(X(\C))$. Thus, $x^{*}\delta_{*}f$ is a bijection for all points $x$. Therefore, $y^{*}f = x^{*}\delta_{*}f$ is a bijection for all $y$.
\\ \\
Conversely, suppose $y^{*}f$ is a bijection for all $y$, i.e., that $x^{*}\delta_{*}f$ is a bijection for all $x$. Then, $\delta_{*}f$ is an isomorphism in $\textnormal{Sh}(X(\C))$. However, by lemma \ref{deltareflects}, $\delta_{*}$ reflects injections and surjections. Thus $f$ is an isomorphism. 
\end{proof}

We are now ready to discuss the Godement resolution, following \cite{rodriguez2}. In fact, we have the following proposition (\cite{rodriguez2}, Proposition 3.3.1):

\begin{prop} Let $D$ be a category closed under products and filtered colimits, and let $x$ be a point of the site $\mathcal{C}$. Then there is an adjoint pair $(x^{*},x_{*})$ of functors
\begin{equation*}
\textnormal{Sh}(\mathcal{C},D) \overset{x^{*}}{\underset{x_{*}}\rightleftarrows}  D,
\end{equation*}
\begin{equation*}
\textnormal{Hom}_{D}(x^{*}\mathcal{F},E) \cong \textnormal{Hom}_{\textnormal{Sh}(\mathcal{C},D)}(\mathcal{F},x_{*}E)
\end{equation*}
\end{prop}
Rodr\'iguez-Gon\'alez and Roig prove the above proposition by noting that we have explicit formulas for $x^{*}$, using filtered colimits, and $x_{*}$, using products. We refer the formulas to their paper (\cite{rodriguez2}, 3.3).

\begin{defn}\label{godementpoints} Let $\mathcal{C}$ be a site with enough points, with the set $S$ as its set of points. Let $D$ be a category closed under products and filtered colimits. By abuse of notation, let $S$ denote the set considered as a discrete category, with only identity morphisms. There is a pair $(p^{*},p_{*})$ of adjoint functors

\begin{equation*}
\textnormal{Sh}(\mathcal{C},D) \overset{p^{*}}{\underset{p_{*}}\rightleftarrows}  D^{S},
\end{equation*}
\begin{equation*}
\textnormal{Hom}_{D^{S}}(p^{*}\mathcal{F},E) \cong \textnormal{Hom}_{\textnormal{Sh}(\mathcal{C},D)}(\mathcal{F},p_{*}E)
\end{equation*}
defined, for $\mathcal{F} \in \textnormal{Sh}(\mathcal{C},D)$ and $E = (E_x)_{x \in S} \in D^S$, by $p^{*}\mathcal{F} := (\mathcal{F}_x)_{x \in S}$ and $p_{*}E := \prod_{x \in S} x_{*}(E_x)$.
\end{defn}

\begin{defn} Let $\mathcal{C}$ be a site with enough points and $D$ a symmetric monoidal category closed under products and filtered colimits. The \textbf{Godement functor} is a functor $G: \textnormal{Sh}(\mathcal{C},D) \rightarrow \textnormal{Sh}(\mathcal{C},D)$ defined as $G = p_{*}p^{*}$ for the adjoint pair $(p^{*},p_{*})$ defined in definition \ref{godementpoints}. The \textbf{cosimplicial Godement functor} is a functor $G^{\bullet}: \textnormal{Sh}(\mathcal{C},D) \rightarrow \Delta\textnormal{Sh}(\mathcal{C},D)$ into cosimplicial sheaves, where the $i$-th term is given by the $i$-th iterate $G^{i}$.
\end{defn}

For the purposes of this paper, we will mainly be concerned with $D = \textnormal{Ch}(k$-$\textnormal{mod})$ the category of cochain complexes of $k$-modules. In this situation, we have the following classical theorem, proven for far more general $D$ by Rodr\'iguez-Gon\'alez and Roig (see \cite{rodriguez} Proposition 3.4.5, Corollary 3.4.6, Proposition 3.4.7) though definitely known to Godement \cite{godement} albeit without modern language. See also Chataur and Cirici (\cite{chataur} Definition 2.5).

\begin{thm}\label{godementismonoidal} Let $f: \mathcal{C} \rightarrow \mathcal{C}'$ be a morphism of sites induced by a continuous functor, where $\mathcal{C}$ and $\mathcal{C}'$ have enough points. Let $D = \textnormal{Ch}(k$-$\textnormal{mod})$. Then the functors $G^{\bullet}$, $f_{*}$, and $\Gamma$ are lax symmetric monoidal. 
\end{thm}

\textbf{Remark.} The symmetric monoidal structure on $\textnormal{Sh}(\mathcal{C},D)$ is by sheafifying the presheaf whose values are tensor products object-wise. That is, given $\mathcal{F}$ and $\mathcal{G}$, we can define a presheaf by $U \mapsto \mathcal{F}(U) \otimes_{k} \mathcal{G}(U)$, and then sheafify. 
\\ \\
We now use some definitions following Chataur and Cirici (\cite{chataur}, Definition 2.4), and Mandell (\cite{mandellCochains}, Proposition 5.2). See also (\cite{stacks}, Tag 019H) for general Dold-Kan considerations.

\begin{defn} The \textbf{normalized complex functor} is the functor $N: \Delta\textnormal{Ch}(k$-$\textnormal{mod}) \rightarrow \textnormal{Ch}(k$-$\textnormal{mod})$ given as the composition of the cosimplicial degree-wise normalization functor and the totalization functor. The \textbf{associated complex functor} is the functor $s: \Delta\textnormal{Ch}(k$-$\textnormal{mod}) \rightarrow \textnormal{Ch}(k$-$\textnormal{mod})$ given by taking alternating sums of the cosimplicial maps to obtain an associated double complex, then applying totalization.
\end{defn}

It is well known that for every cosimplicial complex $C^{\bullet}$, the complexes $NC^{\bullet}$ and $sC^{\bullet}$ are homotopy equivalent (\cite{weibel}, Lemma 8.3.7, Theorem 8.3.8).

\begin{defn} Let $D = \textnormal{Ch}(k$-$\textnormal{mod})$. The \textbf{normalized Godement resolution} is the functor $NG^{\bullet}: \textnormal{Sh}(\mathcal{C},D) \rightarrow \textnormal{Sh}(\mathcal{C},D)$ given as object-wise composition of the normalized complex functor and the cosimplicial Godement functor.
\end{defn}

\begin{defn}
Let $D = \textnormal{Ch}(k$-$\textnormal{mod})$. The \textbf{Godement resolution} is the functor $sG^{\bullet}: \textnormal{Sh}(\mathcal{C},D) \rightarrow \textnormal{Sh}(\mathcal{C},D)$ given as object-wise composition of the associated complex functor and the cosimplicial Godement functor.
\end{defn}

The normalized cochain functor is monoidal but \textit{not} symmetric monoidal; instead we have the following theorem of Hinich and Schechtman \cite{hinich} (see also Chataur and Cirici \cite{chataur}, Proposition 2.2, and Mandell \cite{mandellCochains}, Theorem 5.5).

\begin{thm}\label{hinichschechtman} (Hinich-Schechtman) Let $A^{\bullet}$ be a cosimplicial $\mathcal{P}$-algebra for an arbitrary operad $\mathcal{P}$ in $\textnormal{Ch}(k$-$\textnormal{mod})$. Then $NA^{\bullet}$ is a $(\mathcal{P}\otimes_k \mathcal{Z})$-algebra, for the Eilenberg-Zilbur operad $\mathcal{Z}$, functorial in $A^{\bullet}$ and $\mathcal{P}$.
\end{thm}

We introduce two necessary definitions from \cite{rodriguez} and \cite{rodriguez2} before proceeding. By the remark on the symmetric monoidal structure on $\textnormal{Sh}(\mathcal{C},D)$ via the sheafification functor, the following is well defined. See (\cite{rodriguez}, Remark 3.4.3).

\begin{defn} Let $D = \textnormal{Ch}(k$-$\textnormal{mod})$. A \textbf{sheaf of operads} $\mathcal{P}$ in $D$ is an operad in $\textnormal{Sh}(\mathcal{C},D)$. 
\end{defn}

\begin{defn}\label{sheafOfalgebras} Let $D = \textnormal{Ch}(k$-$\textnormal{mod})$ and let $\mathcal{P}$ be a sheaf of operads in $D$ on a site $\mathcal{C}$. A sheaf of cochain complexes $\mathcal{F} \in \textnormal{Sh}(\mathcal{C},D)$ is a \textbf{sheaf of} $\mathcal{P}$\textbf{-algebras} if it is a $\mathcal{P}$-algebra. Equivalently, for each object $U$ in $\mathcal{C}$, there are the usual structure morphisms for an algebra over an operad $\mathcal{P}(n)(U) \otimes_k \mathcal{F}^{\otimes n}(U) \rightarrow \mathcal{F}(U)$.
\end{defn}

\begin{example} Let $\mathcal{O}$ be a sheaf of $k$-algebras. We can view this as a sheaf of trivial commutative DGAs, which are concentrated in degree 0. Then $\mathcal{O}$ is a sheaf of $Comm$-algebras, where $Comm$ denotes the commutative operad, which has $Comm(n) = k$ in degree 0 for all $n$, with trivial symmetric group actions. Note that $Comm$-algebras are equivalent to commutative DGAs (CDGAs) (\cite{krizMay} Example 2.2). 
\end{example}

\begin{example} Recall that the Eilenberg-Zilbur operad is acyclic (\cite{hinich} Theorem 2.3, \cite{mandellCochains} Proposition 5.4), and so admits a map of operads $\mathcal{Z} \rightarrow Comm$. There is also the fixed map of operads $\mathcal{E} \rightarrow \mathcal{Z}$, where $\mathcal{E}$ is our fixed \Einf-operad. Thus a sheaf $\mathcal{O}$ of $k$-algebras is a sheaf of \Einf-algebras.
\end{example}
\textbf{Remark.} Recall that the symmetric monoidal structure on $\Delta\textnormal{Sh}(\mathcal{C},D)$ is induced degree-wise by the symmetric monoidal structure on $\textnormal{Sh}(\mathcal{C},D)$, so the definitions above make sense for cosimplicial sheaves as well.
\\ \\
\textbf{Remark.} The equivalence above in definition \ref{sheafOfalgebras} follows from the existence of the sheafification functor for presheaves with values in $D = \textnormal{Ch}(k$-$\textnormal{mod})$. Part of the work of Rodr\'iguez-Gon\'alez and Roig in \cite{rodriguez} concerns the conditions on the coefficient category $D$ for which a sheafification functor is available. For the purposes of this paper, we will actually mainly be concerned with the underlying presheaf structure than the sheaf structure per se.

\begin{lemma}\label{CDGAcolimits} The category of $Comm$-algebras, equivalent to the category of commutative differential graded $k$-algebras (CDGAs), is closed under limits and colimits.
\end{lemma}
\begin{proof} Limits and colimits of cochain complexes are computed degree-wise on the underlying $k$-modules, with the differential defined by universal property. A CDGA is precisely a cochain complex $C^{\bullet}$ with a chain map $a: C^{\bullet} \otimes_k C^{\bullet} \rightarrow C^{\bullet}$ unital and associative in the appropriate sense, and satisfying graded commutativity. We sketch a proof of the colimit case in this paragraph and omit the limit case, which is identical. For colimits in CDGAs, we take the colimits of the underlying cochain complexes. Let $D^{\bullet}$ denote the colimit of a system of CDGAs $C^{\bullet}_i \xrightarrow{f_{ji}} C^{\bullet}_j$. We define the algebra map by
\begin{equation*}
D^{\bullet} \otimes_k D^{\bullet} = \textnormal{colim} C^{\bullet}_i \otimes_k \textnormal{colim} C^{\bullet}_j \cong \textnormal{colim} (C^{\bullet}_i \otimes_k C^{\bullet}_i) \xrightarrow{\textnormal{colim } a_i} \textnormal{colim} C^{\bullet}_i = D^{\bullet}
\end{equation*}
where the center natural isomorphism follows from invoking the universal property on $\textnormal{colim } (C^{\bullet}_i \otimes_k C^{\bullet}_i)$ for the system $\{C^{\bullet}_i \otimes_k C^{\bullet}_i\}$, via the maps $C^{\bullet}_i \otimes_k C^{\bullet}_i \rightarrow \textnormal{colim} C^{\bullet}_i \otimes_k \textnormal{colim} C^{\bullet}_j$. For the inverse, one fixes an index $i$ and defines maps $C^{\bullet}_i \otimes_k C^{\bullet}_j \xrightarrow{f_{ji} \otimes \textnormal{id}} C^{\bullet}_j \otimes_k C^{\bullet}_j \rightarrow \textnormal{colim } C^{\bullet}_i \otimes C^{\bullet}_i$ if $i \leq j$, and $C^{\bullet}_i \otimes_k C^{\bullet}_j \xrightarrow{\textnormal{id} \otimes f_{ij}} C^{\bullet}_i \otimes_k C^{\bullet}_i \rightarrow \textnormal{colim } C^{\bullet}_i \otimes C^{\bullet}_i$ if $i > j$; one then uses the universal property twice, and commutativity of all diagrams involved and uniqueness of the maps arising from the universal property verify the isomorphism. Graded commutativity and the Leibniz formula follow from all the terms in the equations being colimits of elements and colimits of maps respectively, and functoriality of colim.

\end{proof}

\begin{lemma}\label{Commlemma} Let $\mathcal{F}$ be a sheaf of $Comm$-algebras and let $f: \mathcal{C} \rightarrow \mathcal{C}'$ be a morphism of sites with enough points induced by a continuous functor $u: \mathcal{C}' \rightarrow \mathcal{C}$. Then $G^{\bullet}\mathcal{F}$ is a cosimplicial presheaf of $Comm$-algebras. If $\mathcal{F}$ is a sheaf of $\mathcal{P}$-algebras for a fixed operad $\mathcal{P}$, then $f_{*}\mathcal{F}$ is a presheaf of $\mathcal{P}$-algebras.
\end{lemma}
\begin{proof} Recall that $G^{0}\mathcal{F} = p_{*}p^{*}\mathcal{F}$. By definition, $G^{0}$ is a product of filtered colimits of $Comm$-algebras, and so is a presheaf of $Comm$-algebras by Lemma \ref{CDGAcolimits}. By definition, $f_{*}\mathcal{F}(U) = \mathcal{F}(u(U))$ for objects $U$ in $\mathcal{C}$; since $\mathcal{F}$ is a presheaf of $\mathcal{P}$-algebras, it follows that $f_{*}\mathcal{F}$ is a presheaf of $\mathcal{P}$-algebras. 
\end{proof}

\textbf{Remark.} By Theorem \ref{godementismonoidal} and results of \cite{rodriguez}, we obtain an induced functor, which we also denote by $G^{\bullet}$, that sends sheaves of operads $\mathcal{P}$ in $\textnormal{Sh}(\mathcal{C},D)$ to cosimplicial sheaves of operads $G^{\bullet}\mathcal{P}$, which is by definition the same as a cosimplicial operad in $\textnormal{Sh}(\mathcal{C},D)$. Similarly, we obtain an induced functor $G^{\bullet}$ that sends sheaves $\mathcal{F}$ of $\mathcal{P}$-algebras to cosimplicial sheaves $G^{\bullet}\mathcal{F}$ which are $G^{\bullet}\mathcal{P}$-algebras, i.e., algebras over the cosimplicial operad $G^{\bullet}\mathcal{P}$.
\\ \\
An immediate consequence of Theorem \ref{godementismonoidal} and Theorem \ref{hinichschechtman} is then the following corollary.

\begin{cor}\label{sheafOfEinf}
Let $D = \textnormal{Ch}(k$-$\textnormal{mod})$ and let $\mathcal{C}$ be a site with enough points with a terminal object $X$. Let $\mathcal{F}$ be a sheaf of $Comm$-algebras. Then $NG^{\bullet}(\mathcal{F})$ is a presheaf of $\mathcal{Z}$-algebras. In particular, \RG$(\mathcal{C},\mathcal{F})$ is an $\mathcal{Z}$-algebra. 
\end{cor}

\begin{proof}
By Theorem \ref{godementismonoidal} and Lemma \ref{Commlemma}, $G^{\bullet}\mathcal{F}$ is a cosimplicial presheaf of $Comm$-algebras. Thus on each object $U$ of $\mathcal{C}$, we have that $NG^{\bullet}(\mathcal{F})(U)$ is the normalized complex of a cosimplicial $Comm$-algebra. By Theorem \ref{hinichschechtman}, $NG^{\bullet}(\mathcal{F})(U)$ is a $(Comm \otimes_k \mathcal{Z})$-algebra for each object $U$. Since $Comm \otimes_k \mathcal{Z} = \mathcal{Z}$, we have that this is a presheaf of $\mathcal{Z}$-algebras. Finally, we have \RG$(\mathcal{C},\mathcal{F}) = NG^{\bullet}(\mathcal{F})(X)$ since the Godement resolution is a resolution by flasque sheaves, which are acyclic. 
\end{proof}

We need just one more theorem, due to Mandell (\cite{mandellCochains}, Theorem 5.8). See also Chataur and Cirici (\cite{chataur}, Definition 2.4).

\begin{thm}\label{mandellnormal} There is a cosimplicial normalization functor $N$ that sends cosimplicial \Einf-algebras to \Einf-algebras, that agrees with the normalized cochain functor $N$ on the underlying cochain complexes, and such that, for a constant cosimplicial \Einf-algebra $A^{\bullet}$, the isomorphism of cochain complexes $A^{0} \cong N(A^{\bullet})$ is a morphism of \Einf-algebras.
\end{thm}

\begin{lemma}\label{secondmap} Let $\mathcal{O}$ be a sheaf of $k$-algebras. The cosimplicial Godement functor has an augmentation $\mathcal{O} \rightarrow G^{\bullet}\mathcal{O}$ which is a map of presheaves of $k$-algebras in degree $0$. The existence of the augmentation is equivalent to the existence of a map of cosimplicial presheaves of $k$-algebras $\mathcal{O}^{\bullet} \rightarrow G^{\bullet}\mathcal{O}$. The composition $\mathcal{O} \xrightarrow{\cong} N\mathcal{O^{\bullet}} \rightarrow NG^{\bullet}\mathcal{O}$ is a map of presheaves of \Einf-algebras.
\end{lemma}
\begin{proof} That an augmentation of a cosimplicial object is equivalent to a map from the constant cosimplicial object is by composing the augmentation map with the cosimplicial maps; see (\cite{stacks} Tag 018F Lemma 14.20.2). The map $\mathcal{O}^{\bullet} \rightarrow G^{\bullet}\mathcal{O}$ of cosimplicial presheaves of $k$-algebras is then a map of cosimplicial presheaves of $Comm$-algebras, and so \Einf-algebras. Applying Theorem \ref{mandellnormal} gives the last statement.
\end{proof}

\begin{lemma} Let $\mathcal{O}$ and $\mathcal{O}'$ be sheaves of $k$-algebras on sites $\mathcal{C}$, $\mathcal{C}'$ respectively. Assume $\mathcal{C}$ and $\mathcal{C}'$ both have terminal objects. A morphism of ringed sites $(\mathcal{C},\mathcal{O}) \xrightarrow{f} (\mathcal{C}',\mathcal{O}') $ induced by a continuous functor $\mathcal{C}' \rightarrow \mathcal{C}$  induces a map of \Einf-algebras \RG$(\mathcal{C}',\mathcal{O}') \rightarrow$ \RG$(\mathcal{C},\mathcal{O})$.  
\end{lemma}
\begin{proof} A morphism of ringed sites by definition has the data of a map of sheaves of rings $\mathcal{O'} \rightarrow f_{*}\mathcal{O}$. Viewing $f_{*}\mathcal{O}$ as an object in the derived category concentrated in degree $0$, there is a natural chain map $f_{*}\mathcal{O} \rightarrow \textnormal{R}f_{*}\mathcal{O}$. By functoriality of $\textnormal{RHom}_{\textnormal{PSh}(\mathcal{C}')}(*_{\mathcal{C}'},-)$ we obtain maps
\begin{equation*}
\textnormal{RHom}_{ \textnormal{PSh}(\mathcal{C}')}(*_{\mathcal{C}'}, \mathcal{O}') \rightarrow \textnormal{RHom}_{\textnormal{PSh}(\mathcal{C}')}(*_{\mathcal{C}'}, f_{*}\mathcal{O}) \rightarrow \textnormal{RHom}_{\textnormal{PSh}(\mathcal{C}')}(*_{\mathcal{C}'}, \textnormal{R}f_{*}\mathcal{O}) \simeq \textnormal{RHom}_{\textnormal{PSh}(\mathcal{C})}(*_{\mathcal{C}}, \mathcal{O})
\end{equation*}
where the last isomorphism follows from Lemma \ref{derivedcomposition}. The composition of these maps gives our map \RG$(\mathcal{C}',\mathcal{O}') \rightarrow $ \RG$(\mathcal{C},\mathcal{O})$. The claim is that this naturally induced map is a map of \Einf-algebras. Note that Theorem \ref{godementismonoidal} also implies that the last isomorphism from Lemma \ref{derivedcomposition} is a quasi-isomorphism of \Einf-algebras, for sheaves of $k$-algebras.
\\ \\
First, note we already have maps of bounded below cochain complexes of sheaves of abelian groups $\mathcal{O'} \rightarrow f_{*}\mathcal{O} \rightarrow \textnormal{R}f_{*}\mathcal{O}$ where we view $\mathcal{O}'$ and $f_{*}\mathcal{O}$ as complexes concentrated in degree $0$. The first map $\mathcal{O}' \rightarrow f_{*}\mathcal{O}$ is a map of presheaves of $k$-algebras, so of $Comm$-algebras, and so of \Einf-algebras. We have that $\textnormal{R}f_{*}\mathcal{O}$ is computed as $f_{*}NG^{\bullet}\mathcal{O}$, which is a presheaf of $\mathcal{Z}$-algebras, and so of \Einf-algebras, by Corollary \ref{sheafOfEinf} and Lemma \ref{Commlemma}. The second map is in fact the functor $f_{*}$ applied to the map $\mathcal{O} \xrightarrow{\cong} N\mathcal{O}^{\bullet} \rightarrow NG^{\bullet}\mathcal{O}$ in Lemma \ref{secondmap}; since this map is a map of presheaves of \Einf-algebras, applying the functor $f_{*}$ yields a map of presheaves of \Einf-algebras. We then have the following commutative square

\begin{center}
\begin{tikzcd}
\mathcal{O'} \arrow[r, "\simeq"]  \arrow[d] & NG^{\bullet}(\mathcal{O'}) \arrow[d, "{*}" right] \\
\textnormal{R}f_{*}\mathcal{O} \arrow[r, "\simeq"] & NG^{\bullet}(\textnormal{R}f_{*}\mathcal{O})
\end{tikzcd}
\end{center}
by functoriality of the Godement resolution on the underlying complexes. The left hand vertical map is a map of presheaves of \Einf-algebras, by the above discussion. The right hand vertical map is then a map of presheaves of $(E_{\infty} \otimes_k \mathcal{Z})$-algebras, and so a map of presheaves of \Einf-algebras; this follows from the existence of maps of operads $\mathcal{E} \rightarrow \mathcal{E} \otimes_k \mathcal{E} \rightarrow \mathcal{E} \otimes_k \mathcal{Z}$. Taking global sections gives us a map of \Einf-algebras 
\begin{equation*}
\textnormal{R}\Gamma(\mathcal{C}',\mathcal{O}') = \Gamma(\mathcal{C}', NG^{\bullet}(\mathcal{O}')) \xrightarrow{*} \Gamma(\mathcal{C}',NG^{\bullet}(\textnormal{R}f_{*}{\mathcal{O}})) = \textnormal{R}\Gamma(\mathcal{C},\mathcal{O})
\end{equation*} 
where the map $*$ is exactly the map described in the first paragraph above, by definition of derived global sections.
\end{proof}

\begin{cor} (Artin comparison) Let $X$ be a smooth proper complex variety. The zig-zag of sites in Theorem \ref{zigzagsites} gives a quasi-isomorphism of \Einf-algebras between \FpEtaleCochains and $\textnormal{R}\Gamma(X(\C),\mathbb{F}_p)$. 
\end{cor}
Recall that the singular cochains of a space $C^{*}_{\textnormal{sing}}(X,k)$ have a natural \Einf-$k$-algebra structure, as an algebra over the Eilenberg-Zilbur operad $\mathcal{Z}_k$; this structure arises from the failure of commutativity of the cup product on cochains with values in a general field $k$, and the Eilenberg-Zilbur operad action is essentially given by the Alexander-Whitney maps (\cite{maySheaf}, Theorem 3.9). This \Einf-algebra structure can be computed again using the Godement resolution. We conclude with the following (again, very classical) statement neatly explained by Petersen in \cite{petersen}, who also assumes the space only be \textit{cohomologically locally connected}.

\begin{thm}\label{ConstantandSing} Let $X$ be a locally contractible, paracompact Hausdorff space. There is a quasi-isomorphism of \Einf-algebras between $\textnormal{R}\Gamma(X, k)$ and $C^{*}_{\textnormal{sing}}(X, k)$. 
\end{thm}
Petersen's proof of the above follows the same strategy as the other proofs in this section, where one passes to the Godement resolution to see that one has a map of \Einf-algebras. 

\section{$p$-adic homotopy theory}
We are now ready to relate the result of Bhatt-Morrow-Scholze to the $p$-adic homotopy theory of the complex variety $X$. We simply apply $-\otimes^L_{\mathbb{F}_p}\overline{\mathbb{F}_p}$ to the Bhatt-Morrow-Scholze result and apply Mandell's functor $\overline{\mathbb{U}}$. First we will briefly review the proof of Theorem \ref{BMSetale}, and then establish some basic change-of-scalars results whose proofs we will refer to Mandell \cite{mandellEinf}, or Kriz and May \cite{krizMay}. Notice that since $\overline{\mathbb{F}_p}$ is a field extension over \Fp, it is faithfully flat. Thus, the functor $-\otimes^L_{\mathbb{F}_p}\overline{\mathbb{F}_p}$ is equal to the functor $-\otimes_{\mathbb{F}_p}\overline{\mathbb{F}_p}$ on the underlying complexes in the derived category. \\ \\
Theorem \ref{BMSetale} can be obtained by first proving the following theorem (\cite{bhatt}, Theorem 9.1) for affine opens of $\mathfrak{X}$, then gluing the result to pass from local to global.

\begin{thm}\label{BMSbigetale} (Bhatt-Scholze) Let $\mathfrak{X} = \textnormal{Spf}(S)$ be a formal affine scheme. There is a canonical quasi-isomorphism of \Einf-algebras 
\begin{center}
\RG$(\textnormal{Spec}(S[1/p],\Z/p^{n}) \simeq (\Prism_{S/A}[1/d]/p^{n})^{\phi=1} $
\end{center}
for each $n \geq 1$.
\end{thm}

\textit{Very rough sketch of proof.} Bhatt and Scholze first prove the two functors
\begin{equation*}
\begin{split}
\textnormal{Spf}(S) \mapsto  \textnormal{R}\Gamma(\textnormal{Spec}(S[1/p],\Z/p^{n}) \\
\textnormal{Spf}(S) \mapsto (\Prism_{S/A}[1/d]/p^{n})^{\phi=1}
\end{split}
\end{equation*}
are sheaves in the arc topology on affine formal schemes. Given this, they assume $S$ is perfectoid, using that affine perfectoids are a basis for the arc topology (\cite{bhatt}, Lemma 8.8). Since $S$ is perfectoid, they can then identify $ (\Prism_{S/A}[1/d]/p^{n})^{\phi=1}$ with $\textnormal{R}\Gamma(\textnormal{Spec}(S^\flat[1/d]),\Z/p^n)$ where $S^\flat$ is the tilt of $S$. The quasi-isomorphism of \Einf-algebras between \RG$(\textnormal{Spec}(S[1/p],\Z/p^{n})$ and  $\textnormal{R}\Gamma(\textnormal{Spec}(S^\flat[1/d]),\Z/p^n)$ then follows from the \etale cohomology comparison between $\textnormal{Spec}(S[1/p])$ and the adic space of the Huber pair $\textnormal{Spa}(S[1/p],S)$ (\cite{huber}, Corollary 3.2.2) (and similarly for $S^\flat[1/d]$), and the equivalence of \etale sites between a perfectoid space and its tilt (\cite{scholze}, Theorem 1.11); this equivalence of \etale sites is then the root of the \Einf-algebra comparison as in section \ref{section2}.  
\\ \\
First, a preliminary lemma of Mandell, whose proof we refer to \cite{mandellEinf}, in the paragraph preceding Proposition A.8.

\begin{lemma}\label{extensionScalars} There is a map of operads $\mathcal{E}_{\mathbb{F}_p} \otimes_{\mathbb{F}_p} \overline{\mathbb{F}_p} \rightarrow \mathcal{E}_{\overline{\mathbb{F}_p}}$. By changing the operad $\mathcal{E}_{\overline{\mathbb{F}_p}}$ if necessary, this map is an isomorphism. This defines a functor from \Einf-\Fp-algebras to \Einf-\Fpbar-algebras that sends $A$ to $A \otimes_{\mathbb{F}_p} \overline{\mathbb{F}_p}$.
\end{lemma}

Now we study the extended scalars $C^{*}_{\textnormal{sing}}(X,\mathbb{F}_p) \otimes_{\mathbb{F}_p} \overline{\mathbb{F}_p}$. This is an \Einf-\Fpbar-algebra by the above lemma. However, there is another \Einf-\Fpbar-algebra structure as described in the lemma below; we will prove that these two structures agree. See Appendix B in \cite{mandellEinf} for details.
\begin{lemma}\label{secondExtensionScalars} $C^{*}_{\textnormal{sing}}(X,\mathbb{F}_p) \otimes_{\mathbb{F}_p} \overline{\mathbb{F}_p}$ is an algebra over $\mathcal{E}_{\overline{\mathbb{F}_p}}$. There is a natural map of $\mathcal{E}_{\overline{\mathbb{F}_p}}$-algebras $C^{*}_{\textnormal{sing}}(X,\mathbb{F}_p) \otimes_{\mathbb{F}_p} \overline{\mathbb{F}_p} \rightarrow C^{*}_{\textnormal{sing}}(X,\overline{\mathbb{F}_p})$.
\end{lemma}
\begin{proof} Let $\{ X_{\alpha} \} $ denote the inverse system of levelwise finite quotients of $X$; that is, $X_{\alpha}$ is a quotient of $X$ as a simplicial set, and in each simplicial degree $n$, the set of $n$ simplices is finite. We give $C^{*}_{\textnormal{sing}}(X,\mathbb{F}_p) \otimes_{\mathbb{F}_p} \overline{\mathbb{F}_p}$ the structure of an \Einf-$\overline{\mathbb{F}_p}$-algebra via the natural isomorphism $C^{*}_{\textnormal{sing}}(X,\mathbb{F}_p) \otimes_{\mathbb{F}_p} \overline{\mathbb{F}_p} \cong \textnormal{colim}_{\alpha}$ $C^{*}(X_{\alpha},\overline{\mathbb{F}_p})$. The inverse system of maps $X \rightarrow X_{\alpha}$ induces a map of \Einf-$\overline{\mathbb{F}_p}$-algebras $C^{*}_{\textnormal{sing}}(X,\mathbb{F}_p)\otimes_{\mathbb{F}_p}\overline{\mathbb{F}_p} \rightarrow C^{*}_{\textnormal{sing}}(X,\overline{\mathbb{F}_p})$. 
\end{proof}

The following argument was communicated to the author by Mandell.

\begin{lemma} The \Einf-\Fpbar-algebra structure on $C^{*}_{\textnormal{sing}}(X,\mathbb{F}_p) \otimes_{\mathbb{F}_p} \overline{\mathbb{F}_p}$ from Lemma \ref{extensionScalars} and the \Einf-\Fpbar-algebra structure from Lemma \ref{secondExtensionScalars} agree. 
\end{lemma}
\begin{proof} The maps $X \rightarrow X_{\alpha}$ induce a map of \Einf-\Fpbar-algebras in the sense of Lemma \ref{extensionScalars}
\begin{equation*} 
C^{*}_{\textnormal{sing}}(X_{\alpha},\mathbb{F}_p)\otimes_{\mathbb{F}_p}\overline{\mathbb{F}_p} \rightarrow C^{*}_{\textnormal{sing}}(X,\mathbb{F}_p)\otimes_{\mathbb{F}_p}\overline{\mathbb{F}_p}.
\end{equation*}
There is a natural map $C^{*}_{\textnormal{sing}}(X_{\alpha},\mathbb{F}_p) \otimes_{\mathbb{F}_p} \overline{\mathbb{F}_p} \rightarrow C^{*}_{\textnormal{sing}}(X_{\alpha},\overline{\mathbb{F}_p})$ of \Einf-\Fpbar-algebras, which is an isomorphism since $X_{\alpha}$ is degreewise finite; that is, the lemma holds for $X_{\alpha}$ by finiteness. Thus we obtain a system of maps of \Einf-\Fpbar-algebras
\begin{equation*}
C^{*}_{\textnormal{sing}}(X_{\alpha},\overline{\mathbb{F}_p}) \rightarrow C^{*}_{\textnormal{sing}}(X,\mathbb{F}_p)\otimes_{\mathbb{F}_p}\overline{\mathbb{F}_p}.
\end{equation*} 
This system is filtered, so the colimit over $\alpha$ of $C^{*}(X_{\alpha},\overline{\mathbb{F}_p})$ in the category of \Einf-\Fpbar-algebras is the same colimit on the underlying differential graded modules over \Fpbar. By universal property of colimits in \Einf-\Fpbar-algebras, we obtain an \Einf-\Fpbar-algebra map 
\begin{equation*}
\textnormal{colim}_{\alpha } C^{*}_{\textnormal{sing}}(X_{\alpha},\overline{\mathbb{F}_p}) \rightarrow C^{*}_{\textnormal{sing}}(X,\mathbb{F}_p)\otimes_{\mathbb{F}_p}\overline{\mathbb{F}_p}
\end{equation*}
which is an isomorphism on the underlying differential graded modules over \Fpbar, and so is an isomorphism of \Einf-\Fpbar-algebras. The \Einf-\Fpbar-algebra structure on the colimit is in the sense of Lemma \ref{secondExtensionScalars}.
\end{proof}

The following theorem of Mandell (\cite{mandellEinf}, Theorem B.1) relates the above discussion to the $p$-profinite completion of the variety $X$ in the sense of Sullivan (\cite{sullivan}, Section 3). 

\begin{thm}\label{mandellTheorem3} For any connected simplicial set $X$, the composite map
\begin{equation*}
X \rightarrow \overline{\mathbb{U}}C^{*}_{\textnormal{sing}}(X,\overline{\mathbb{F}_p}) \rightarrow \overline{\mathbb{U}}(C^{*}_{\textnormal{sing}}(X,\mathbb{F}_p)\otimes_{\mathbb{F}_p}\overline{\mathbb{F}_p})
\end{equation*}
is Sullivan $p$-completion.
\end{thm}
The first map is from the unit of the adjunction in Theorem \ref{mandellTheorem1}. By Lemma \ref{secondExtensionScalars}, we have a map $C^{*}_{\textnormal{sing}}(X,\mathbb{F}_p) \otimes_{\mathbb{F}_p} \overline{\mathbb{F}_p} \rightarrow C^{*}_{\textnormal{sing}}(X,\overline{\mathbb{F}_p})$. This induces the second map above.
\\ \\
We are now ready to prove the main theorem \ref{MAINTHEOREM}.
\\ \\
\textit{Proof of Theorem} \ref{MAINTHEOREM}. The sheaf cohomology of any sheaf on the \etale site of $X$ inherits an \Einf-algebra structure by Godement considerations, as in section \ref{section2}. By Bhatt-Morrow-Scholze \ref{BMSetale} and Theorem \ref{BMSbigetale}, we have that
\begin{equation*}
(\textnormal{R}\Gamma(\frak{X},\Prism_{\frak{X}/A_{\textnormal{inf}}})\otimes^{L}_{A_{\textnormal{inf}}} \mathbb{C}^{\flat}_p)^{\phi=1} \simeq \textnormal{R}\Gamma_{\textnormal{\'et}}(X,\mathbb{F}_p)
\end{equation*}  
is a quasi-isomorphism of \Einf-\Fp-algebras. By the entirety of section \ref{section2}, the right hand side is quasi-isomorphic to $\textnormal{R}\Gamma(X(\C),\mathbb{F}_p)$ as \Einf-\Fp-algebras. By Theorem \ref{ConstantandSing}, we again obtain a quasi-isomorphism of \Einf-\Fp-algebras with $C^{*}_{\textnormal{sing}}(X(\C),\mathbb{F}_p)$. By Mandell \ref{mandellTheorem2}, applying the functor $\mathbb{U}$ to the left hand side then yields the free loop space of the Bousfield-Kan $p$-completion of $X(\C)$.
\\ \\
On the other hand, we can apply $-\otimes^L_{\mathbb{F}_p} \overline{\mathbb{F}_p}$ to the equation above. Again by Mandell's Theorem \ref{mandellTheorem1} and Theorem \ref{mandellTheorem3}, applying the functor $\overline{\mathbb{U}}$ to the resulting expression gives the Sullivan $p$-completion of $X(\C)$. $\qed$ 
\\ \\
\textbf{Remark.} We have seen that applying Mandell's functors $\mathbb{U}$ and $\overline{\mathbb{U}}$ to the Bhatt-Morrow-Scholze result recovers the free loop space of the Bousfield-Kan $p$-completion, and the Sullivan $p$-completion respectively. One may wonder if the Bousfield-Kan $p$-completion itself is accessible, without the presence of the free loop space. The author suspects this is doable given the following: Mandell proves Theorem \ref{mandellTheorem2} by taking homotopy fixed points of the Frobenius on $A\otimes_{\mathbb{F}_p}\overline{\mathbb{F}_p}$ where $A$ is a cofibrant replacement of $C^{*}_{\textnormal{sing}}(X,\mathbb{F}_p)$. Mandell shows this is weak equivalent (for $X$ connected, $p$-complete, nilpotent, and of finite $p$-type) to the homotopy fixed points of the space $X$ with the trivial action, which yields the free loop space. At the same time, the theorem of Bhatt-Morrow-Scholze recovers the \etale \Fp-cochains by taking homotopy fixed points of the Frobenius on $(\textnormal{R}\Gamma(\frak{X},\Prism_{\frak{X}/A_{\textnormal{inf}}})\otimes^{L}_{A_{\textnormal{inf}}} \mathbb{C}^{\flat}_p)$. However, this latter object is enormous over \Fpbar. The author suspects by passing to spectra and analyzing an appropriate Frobenius, as in \cite{nikolaus} or \cite{yuan}, and using the relationship between prismatic cohomology and topological Hothschild homology, one could recover a clean statement regarding the Bousfield-Kan $p$-completion of the variety $X$ over $\C$.

\textsc{Stony Brook University, Department of Mathematics}
\\
\textit{E-mail address:} {\fontfamily{cmtt}\selectfont tobias.shin@stonybrook.edu}

\end{document}